\documentclass{amsart}
\usepackage{amssymb, latexsym} 

\newtheorem{theorem}{Theorem}

\newtheorem{claim}{Claim}

\begin{document}
\title{On Growth of Double Cosets in Hyperbolic  Groups}
\author{Rita Gitik}
\address{ Department of Mathematics \\ University of Michigan \\ Ann Arbor, MI, 48109}
\email{ritagtk@umich.edu}
\author{Eliyahu Rips}
\address{ Institute of Mathematics \\ Hebrew University, Jerusalem, 91904, Israel}
\email{eliyahu.rips@mail.huji.ac.il} 

\date{\today}

\begin{abstract}
Let $H$ be a hyperbolic group, $A$ and $B$ be subgroups of $H$, and $gr(H,A,B)$ be the growth function of the double cosets $AhB, h \in H$. We prove that the behavior of $gr(H,A,B)$ splits into two different cases. If $A$ and $B$ are not quasiconvex, we obtain that every growth function of a finitely presented group can appear as $gr(H,A,B)$. We can even take $A=B$. In contrast, for quasiconvex subgroups A and B of infinite index, $gr(H,A,B)$ is exponential. Moreover, there exists a constant $\lambda > 0$, such that $gr(H,A,B)(r) > \lambda  f_H(r)$ for all big enough $r$, where $f_H(r)$ is the growth function of the group $H$. So, we have a clear dychotomy between the quasiconvex and non-quasiconvex case.
\end{abstract}


\maketitle

\section{Introduction}

Growth of groups has been a subject of research for many years. For main results and references see \cite{Ha}. Growth of cosets and related subsets in groups has also been investigated, see for example \cite{Ol} and \cite{C-K}. However, de la Harp wrote in \cite{Ha}, p.209 that growth of double cosets in groups has not yet received much attention, but probably should. In this paper we investigate growth of double cosets of hyperbolic groups.
 
Let $H$ be a hyperbolic group and let $A$ and $B$ be finitely generated subgroups of $H$.
Fix some set of generators of $H$. For any $k \ge 0$,  let $gr(H,A,B)(r)$ be the growth function for double cosets $AhB$, that is 
$gr(H,A,B)(r)=| \{ AhB, |h| \le r \}|$, where $|h|$ is the length of $h$.

Our first theorem shows that the class of growth rate functions of double cosets of non-quasiconvex subgroups is rather wide. Namely, let $G=\langle x_1, \cdots , x_m|r_1, \cdots , r_n \rangle$ be any finitely presented group. Let $f_G$ be the growth function of $G$, that is $f_G(r)=|\{g \in G, |g| \le r \}|$, where the length $|g|$ is taken with respect to the set of generators
 $x_1, \cdots, x_m$.

\begin{theorem}
There exists a hyperbolic group $H$ and a finitely generated subgroup $N$ of $H$ such that $gr(H,N,N)=f_G$.
\end{theorem}

Our second theorem shows that if $A$ and $B$ are quasiconvex then the growth function of the double cosets $gr(H,A,B)$ is exponential.

\begin{theorem}
Let $H$ be a non-elementary hyperbolic group with a fixed set of generators.
Let $A$ and $B$ be quasiconvex subgroups of $H$ of infinite index.  Then there exists a constant 
$\lambda >0$ such that $gr(H,A,B)(r) \ge \lambda f_H(r)$ for all big enough $r$.
\end{theorem}

\section{Proof of The First Theorem}

Let $G=\langle x_1, \cdots , x_m|r_1, \cdots , r_n \rangle$ be a finitely presented group.
By a well-known construction of the second author \cite{Rips}, there exists a hyperbolic group $H$ and a short exact sequence
$1 \longrightarrow N \overset{\alpha}{\longrightarrow}   H   \overset{\beta}{\longrightarrow}   G \longrightarrow 1$ such that the normal subgroup $N$ of $H$ is finitely generated
as a group. Indeed, following \cite{Rips}, define $H$ to be a group generated by the elements $x_1, \cdots , x_m, t_1, t_2$ with defining relations

$r_it_1t_2^{a_i}t_1t_2^{a_i +1} \cdots t_1t_2^{b_i} \; \;(i=1,2, \cdots ,n)$,

$x_i^{-1}t_jx_it_1t_2^{c_{ij}}t_1t_2^{c_{ij+1}} \cdots t_1t_2^{d_{ij}} \; \; (i=1,2,\cdots ,m, j=1,2)$,

$x_it_jx_i^{-1}t_1t_2^{e_{ij}}t_1t_2^{e_{ij+1}} \cdots t_1t_2^{f_{ij}} \; \; (i=1,2, \cdots ,m,j=1,2)$.

We can choose the constants $a_i,b_i,c_{ij},d_{ij},e_{ij},f_{ij}$ such that $H$ is a small cancellation group with arbitrarily small constant,
and hence $H$ is a hyperbolic group.

Let us show that $gr(H,N,N)=f_G$.

For every double coset $NhN$ consider the element $\beta(h) \in G$. This is a $1$-to-$1$ correspondence between double cosets and elements of $G$. Indeed, $Nh_1N=Nh_2N$ implies $\beta(h_1)=\beta(h_2)$. Also for every 
$g=x_{i_1}^{\epsilon_1}x_{i_2}^{\epsilon_2} \cdots x_{i_s}^{\epsilon_s} \in G$,  $\beta(g)=g$, so to $NgN$ corresponds the element $\beta(g)=g \in G$.

The homomorphism $\beta:H \longrightarrow G$ sends $x_i$ to $x_i$ and $t_j$ to $1$ for $i=1,2, \cdots ,m, j=1,2$. It follows that for every $h \in H$,
$|\beta(h)| \le |h|$, where the length of $h$ is taken with respect to the generators $x_1, \cdots, x_m, t_1,t_2$ of $H$ and the length of $\beta(h)$ is taken with respect to the generators $x_1, \cdots ,x_m$ of $G$. 

Therefore, for any $k \ge 0$, the number of double cosets $NhN$ with $|h| \le k$ is equal to the number of elements $g \in G$ with $|g| \le k$. Hence,
$gr(H,N,N)=f_G$, proving Theorem 1.

\section{Proof of The Second Theorem}

\begin{claim}
Let $A$ be an infinite index quasiconvex subgroup of a hyperbolic group $H$. There exists an element $c_1 \in H$ of infinite order such that 
$A \cap \langle c_1 \rangle = \{ 1 \}$.
\end{claim}
\begin{proof}
Denote the boundary of $H$ by $\partial H$.
Note that the limit set $\Lambda(A)$ of $A$ in $\partial H$ is a closed subset. As $A$ is infinite and quasiconvex, it has finite index in the setwise stabilizer $Stab_H(\Lambda(A))$. In particular, this implies that $\Lambda(A) \neq \partial H$ because if $\Lambda(A) = \partial H$ then 
$Stab_H(\Lambda(A))=H$, contradicting the assumption that $A$ has infinite index in $H$. Hence
$\Omega = \partial H - \Lambda(A)$ is a nonempty open subset of $\partial H$.

Since $H$ is hyperbolic, the set of poles $t_+$ of all infinite order elements $t \in H$ is dense in
$\partial H$. Hence there exists an element $c_1 \in H$ of infinite order such that $c_{1^+} \in \Omega$. Therefore  $A \cap \langle c_1 \rangle = \{ 1 \}$, since otherwise $c_1^n \in A$ for some $n >0$ and $c_{1^+} \in \Lambda(A)$, contradicting the assumption that $c_{1^+} \in \Omega$.
\end{proof}

The only facts used in the proof of Claim 1 are:

\begin{enumerate}
\item Poles of infinite order elements are dense in the boundary of the group $H$.

\item If $A$ is an infinite quasiconvex subgroup of $H$ then $A$ has finite index in  $Stab_H(\Lambda(A))$.
\end{enumerate}

The aforementioned facts hold for an arbitrary hyperbolic group $H$, without any restrictions on the existence of nontrivial torsion elements.

\bigskip

\begin{claim}
Let $H$ be a non-elementary hyperbolic group, let $A$ be a quasiconvex subgroup of infinite index in $H$, and let $c_1 \in H$ be an element of infinite order such that $A \cap \langle c_1 \rangle = \{1 \}$. Denote by $C_1$ the commensurator of $\langle c_1 \rangle$ in $H$ and let $C_2 = A \cap C_1$. 
There exists a large integer $M$ such that for  $C= \langle c_1^M, C_2 \rangle$ the following holds:
$ \langle A, C \rangle = A \underset{C_2}* C$ and $ \langle A, C \rangle$ is a quasiconvex subgroup of $H$.
\end{claim}

Claim 2 follows from Theorem 1.1 in \cite{MP}, which generalizes Theorem 2 in \cite{Gi1}, because any non-elementary hyperbolic group $G$ is hyperbolic relative to the commensurator of an infinite cyclic subgroup. That was observed in Theorem 7.11 in \cite{Bo}.

\bigskip

\begin{claim}
Let $H$ be a non-elementary hyperbolic group with a fixed generating set and let $(\epsilon_0, \eta_0)$ be constants.
There exist geodesic words $x$ and $y$, and constants $(\epsilon, \eta)$ such that for any 
$(\epsilon_0, \eta_0)$-quasigeodesic words $u$ and $v$ in $H$ at least one of the words $\{ uxv, ux^{-1}v, uyv, uy^{-1}v \}$
is an $(\epsilon, \eta)$-quasigeodesic.
\end{claim}
\begin{proof}
As we are interested only in the existence of the constants $(\epsilon, \eta)$ we do not have to compute them explicitly.
Let $\delta$ be a hyperbolicity constant of $H$. Denote the Gromov product of the words $w$ and $z$ by $(w,z)_1$, cf. \cite{Gr} and \cite{B-K} definition 2.6. 

For a big enough constant $C_0$ there exist geodesic words $x$ and $y$ which are longer than $C_0$, such that $(w,z)_1 \le \beta$, where $\beta$ is a constant bigger than $\delta$, and the words $w$ and $z$ belong to the set $\{ x, x^{-1}, y, y^{-1} \}$.

Note that for any  $(\epsilon_0, \eta_0)$-quasigeodesic $u$ only one of the Gromov products $\{ (u,x)_1, (u,x^{-1})_1, (u,y)_1, (u,y^{-1})_1 \}$ might be big. Similarly, only one of the Gromov products $\{ (v^{-1},x)_1, (v^{-1},x^{-1})_1, (v^{-1},y)_1, (v^{-1},y^{-1})_1 \}$ might be big. Therefore, at least one of the following conditions holds.
\begin{enumerate}
\item $(u,x)_1$ and $(v^{-1}, x^{-1})_1$ are small;
\item $(u,x^{-1})_1$ and $(v^{-1}, x)_1$ are small;
\item $(u,y)_1$ and $(v^{-1}, y^{-1})_1$ are small;
\item $(u,y^{-1})_1$ and $(v^{-1}, y)_1$ are small.
\end{enumerate}

In case $(1)$ $uxv$ is a quasigeodesic, in case $(2)$ $ux^{-1}v$ is a quasigeodesic, in case $(3)$ $uyv$ is a quasigeodesic, in case $(4)$ $uy^{-1}v$ is a quasigeodesic. Note that the quasigeodesicity constants $(\epsilon, \eta)$ do not depend on $u$ and $v$.
\end{proof}

\bigskip      

\textbf{Proof of Theorem 2.}

Since $A$ and $B$ are quasiconvex and of infinite index in $H$, Claim 1 implies that there exist elements $c_1 \in H$ and $d_1 \in H$ of infinite order such that $ \langle c_1 \rangle \cap A = \{ 1 \}$ and $ \langle d_1 \rangle \cap B = \{ 1 \}$. 
Denote by $C_1$ the commensurator of $\langle c_1 \rangle $ in $H$ and by $D_1$ the commensurator of $\langle d_1 \rangle $ in $H$. Let $C_2=C_1 \cap A$ and $D_2=D_1 \cap B$. Claim 2 implies that there exists a large integer $M$ such that the subgroups $C= \langle c_1^M, C_2 \rangle $ and
$D= \langle d_1^M, D_2 \rangle$ have the following properties:
$ \langle A, C \rangle = A \underset{C_2} * C$ and $ \langle B, D \rangle = B \underset{D_2} * D$.
Denote $c=c_1^M$ and $d=d_1^M$.

According to Claim 3, there exists a large integer $N$ such that for a long enough geodesic word $t$ there exist  $z$ and $w$ in the set
$\{ x, x^{-1}, y, y^{-1} \}$ such that $s=c^Nztwd^N$ is $(\epsilon, \eta)$-quasigeodesic. Note that $|s| \le N|c| + N|d| + |z| +|w| +|t| \le N(|c| +|d|) + 2|x| + 2 |y| + |t|$. So if $|t| \le r -( N(|c| +|d|) + 2|x| + 2 |y| + |t|)$ then $|s| \le r$. As  $w$ and $z$ belong to the set $\{ x, x^{-1}, y, y^{-1} \}$, no more than $16$ different choices of $t$ can produce the same word $s$.

Let  $m = f_H(r -N(|c| +|d|) + 2|x| + 2 |y| - f_H(C_0))$. There exist geodesic words $t_1, t_2, \cdots , t_m$ such that $C_0 < |t_i| \le  r - N(|c| +|d| + 2|x| + 2 |y|)$ for $1 \le i \le m$. It follows that the corresponding words $s_i=c^Nz_it_iw_id^N$ are $(\epsilon, \eta)$-quasigeodesics and $z_i$ and $w_i$ belong to the set $\{ x, x^{-1}, y, y^{-1} \}$. As $f_H$ is the growth function of a non-elementary hyperbolic group, there exists a positive constant $\mu$ such that $16 f_H(r -(N(|c| +|d|) + 2|x| + 2 |y|)) \ge \mu f_H(r)$. 

Consider the double cosets $\{ As_iB| 1 \le i \le m \}$. If $As_iB=As_jB, i \neq j$,
then there exist geodesic words $a \in A$ and $b \in B$ such that $s_j=_H as_ib$ in $H$, hence $b=_H s_i^{-1}a^{-1}s_j=_H d^{-N}w_i^{-1}  t_i^{-1} z_i^{-1}c^{-N}a c^N z_jt_jw_j d^N$. 

If $ a \notin C_2$ then, according to Theorem 1.1 of \cite{MP}, the word $c^{-N}a^{-1}c^{N}$ is a quasigeodesic. As $d^{-N}w_i^{-1}t_i^{-1}z_i^{-1}c_0^{-N}$ and $c^{N}z_jt_jw_jd^N$ are quasigeodesics, it follows that
$d^{-N}w_i^{-1}t_i^{-1}z_i^{-1}c^{-N}a^{-1}c^N z_jt_jw_jd^N$ is also a quasigeodesic, contradicting the fact that it is equal to $b \in B$ in $H$, where the subgroup $B$ is quasiconvex. Similarly, it cannot happen that $b \notin D_2$. Hence $s_j=_H as_ib$ implies that $a \in C_2$ and $b \in D_2$.

Therefore at least $\frac{m}{|C_2| \cdot |D_2|}$ of the double cosets $As_iB, i=1, \cdots, m$ are different from each other. So
$gr(H,A,B)(r) \ge \frac{\mu}{|C_2| \cdot |D_2|} \cdot f_H(r)$ and $\lambda = \frac{\mu}{|C_2| \cdot |D_2|}$ satisfies the requirements of Theorem 2.

\bigskip

\textbf{Corollary to the Proof of Theorem 2.}

Using the notation of the Proof of Theorem 2, for each $i = 1, 2, ,..., m$ the subgroup 
$\langle A, s_i B s_i^{-1} \rangle $ is a quasiconvex subgroup of $H$, and its geodesic core (see  Definition $4$ in \cite{Gi2}) consists of a cylindrical neighborhood of the path for $s_i$ with geodesic cores of $A$ and of $B$ attached at the beginning and at the end of the path for $s_i$ respectively.

\section{Acknowledgment}
We would like to thank Eduardo Martinez-Pedroza for bringing to our attention his paper \cite{MP} and to Ilya Kapovich for helpful discussions.

\end{document}